\documentclass[11pt]{article}

\usepackage[margin=0.9in]{geometry} 
\usepackage{amsmath,amsthm,amssymb,mathrsfs,bbm,titling,mathtools,hyperref,nameref,esint,stmaryrd,enumitem,subcaption,stackengine}
\stackMath
\usepackage{color}
\usepackage{fancyhdr}
\usepackage[style=numeric-comp,backend=biber]{biblatex}
\let\oldref\ref
\renewcommand{\ref}[1]{(\oldref{#1})}

\newcommand{\N}{\mathbb{N}}
\newcommand{\Z}{\mathbb{Z}}
\newcommand{\C}{\mathbb{C}}
\newcommand{\R}{\mathbb{R}}
\newcommand{\Q}{\mathbb{Q}}

\renewcommand{\L}{\mathbb{L}}
\newcommand{\eps}{\varepsilon}

\def\XXint#1#2#3{{\setbox0=\hbox{$#1{#2#3}{\int}$ }
\vcenter{\hbox{$#2#3$ }}\kern-.6\wd0}}

\renewcommand{\O}{\mathcal{O}}

\renewcommand{\cal}{\mathcal}
\newcommand{\bb}{\mathbb}
\renewcommand{\bf}{\mathbf}
\renewcommand{\frak}{\mathfrak}

\newtheorem{theorem}{Theorem}[section]

\newtheorem{corollary}[theorem]{Corollary}

\theoremstyle{definition}
\newtheorem{definition}[theorem]{Definition}
\newtheorem{example}[theorem]{Example}

\newtheorem{problem}[theorem]{Problem}

\theoremstyle{remark}
\newtheorem{remark}[theorem]{Remark}

\numberwithin{equation}{section}

\title{Sparse mean value estimates, algebraic number solution counting, and non-Archimedean Fourier analysis}
\author{Ben Johnsrude}

\begin{document}

	\maketitle
	
	\begin{abstract}
		We demonstrate two applications of Fourier decoupling theorems over non-Archimedean local fields to real-variable problems. These include short mean value estimates for exponential sums, canonical-scale mean value estimates for exponential sums arising from phase functions with coefficients arising from the traces of powers of algebraic numbers, and solution counting bounds for Vinogradov systems whose indeterminates are families of algebraic numbers. We also record an example where real and $\frak p$-adic decoupling estimates differ.
	\end{abstract}

    \section{Introduction}

    In this note, we record two consequences of transferring discrete restriction-type estimates from non-Archimedean local fields to the real numbers. In recent years, the techniques of Fourier-analytic decoupling have been found to transfer easily to $\frak p$-adic and other non-Archimedean local fields. Moreover, those techniques are typically at their most potent in the non-Archimedean setting, where many approximate statements from the reals become literal, e.g. statements regarding essential support on frequency/physical space. On the other hand, there have been limited occasions for the results of non-Archimedean Fourier analysis to imply results over the real numbers, save for the (very important) exception of Diophantine equation solution counting.

     In Section \ref{sec:padic}, we will restrict attention to the $\frak p$-adics and convert short mean value estimates over $\Q_\frak p$ to mean value estimates over $\R$ using the same phase, whose domain is now taken to be a subset of $[0,1]^k$ that is ``$\frak p$-adically shaped'' in a suitable sense. We briefly record some of the state of knowledge for short mean value estimates over the reals. \cite{DGW} conjectured that, for each $n\geq 2,0\leq\beta\leq n-1,$ and $s\geq 1$, the exponential sum associated to the moment curve phase function has short mean value estimate
     \begin{equation}\label{conj:mom_small}
         \int_{[0,1]^{n-1}\times[0,N^{-\beta}]}\left|\sum_{k=1}^Ne(kx_1+k^2x_2+\ldots+k^nx_n)\right|^{2s}dx\lesssim_\eps N^\eps(N^{s-\beta}+N^{2s-\frac{n(n+1)}{2}}),
     \end{equation}
     for every $\eps>0$. Here and elsewhere we write $e(r)=e^{2\pi ir}$ for real numbers $r$. The modifier ``short'' for mean value estimates refers to the shape of the domain $[0,1]^{n-1}\times[0,N^{-\beta}]$, which is shorter than the fundamental domain $[0,1]^n$. The same paper validated the estimate \eqref{conj:mom_small} in the cases $n=2$, and $n=3$ in the restricted range $0\leq\beta\leq\frac{3}{2}$. \cite{guth2022small} resolved the $n=3$ case for the full $0\leq\beta\leq 2$ range. \cite{maldague2024small} proved a sharp estimate in the additional range $2<\beta$, with the modified inequality
     \begin{equation*}
        \begin{split}
             \int_{[0,1]^{n-1}\times[0,N^{-\beta}]}&\left|\sum_{k=1}^Na_ke(kx_1+k^2x_2+\ldots+k^nx_n)\right|^{2s}dx\\
             &\leq C_\eps N^\eps\sum_k|a_k|^{2s}\times\begin{cases}
                 N^{\frac{1}{2}-\frac{\beta+1}{p}}+N^{1-\frac{7}{p}} & \text{if }\,\,\,\,0\leq\beta<\frac{5}{2},  \\
                 N^{\frac{1}{2}-\frac{\beta+1}{p}}+N^{\frac{\beta}{3}-\frac{7\beta}{3p}}+N^{1-\frac{7}{p}} & \text{if }\,\,\,\,\frac{5}{2}\leq\beta<3,  \\
                 N^{\frac{1}{2}-\frac{\beta}{p}}+N^{1-\frac{3+\beta}{p}} & \text{if }\,\,\,\,3\leq\beta  ,
             \end{cases}
         \end{split}
     \end{equation*}
     valid for any $\eps>0$ and any complex coefficients $\{a_k\}_{k=1}^N$.

     In a different direction, \cite{DLexpsum} conjectured the estimate
     \begin{equation}\label{conj:DL}
         2^{j\frac{d+1}{2}}\int_{[0,2^{-j}]^d}\left|\sum_{k=1}^Na_ke(kx_1+k^2x_2+\ldots+k^nx_n)\right|^{p}dx\leq C_\eps N^\eps(1+N^{\frac{p-\rho_d}{2}})\Big(\sum_k|a_k|^2\Big)^{p/2},
     \end{equation}
     where
     \begin{equation*}
         \rho_d=\begin{cases}
             \frac{3d^2-4}{4} & d\text{ even},\\
             \frac{3d^2-3}{4} & d\text{ odd}.
         \end{cases}
     \end{equation*}
     The same paper validated \eqref{conj:DL} when $d=2$ or $d=3$, and in the restricted range $p\leq 10$ for $d=4$.
     
     Similar short mean value estimates have been studied for other phase functions, e.g. the sphere, cone, and paraboloid. We discuss the latter case below.

     Our ``$\frak p$-adically shaped'' short mean value estimates are over domains that are short in a $\frak p$-adic sense, i.e. which are sparsely separated. The general theme is to compare the ``Archimedean interval''
     \begin{equation*}
         \{0,1,2,\ldots,\frak p^k-1\}\subseteq\{0,1,2,\ldots, \frak p^N-1\}
     \end{equation*}
     with the ``non-Archimedean interval''
      \begin{equation*}
          \{0,\frak p^{N-k},2\frak p^{N-k},\ldots,\frak p^N-\frak p^k\}\subseteq\{0,1,2,\ldots, \frak p^N-1\}.
      \end{equation*}
      Short mean value estimates for $\frak p$-adic exponential sums can be naturally interpreted as mean value estimates for real exponential sums over sparse domains; the precise statement we obtain is Theorem \ref{thm:transfer} below.
     
     In Section \ref{sec:numfield}, we will move to non-Archimedean completions of number fields and conclude both mean value estimates and solution-counting estimates from the decoupling theorem for the moment curve over non-Archimedean local fields. In the case of mean value estimates, we conclude bounds for real integrals with polynomial phases incorporating coefficients arising from the traces $\mathrm{Tr}_{\Q(\alpha)/\Q}(\alpha^\kappa)$ for various algebraic numbers $\alpha$ and nonnegative integers $\kappa$. The precise statement is Theorem \ref{thm:mv_nf}. In the case of solution-counting, the indeterminates are generalized from integers to algebraic numbers. The precise statement is Corollary \ref{cor:alg_sol}. We are not aware of these results appearing elsewhere in the literature. In both cases, our results may arise from simply considering integrals over fundamental domains and using even integer exponents and the sum-product expansion.

     In Section \ref{section:obstruction}, we record an example where the $\frak p$-adic decoupling theory differs significantly from the real decoupling theory, preventing a matching estimate from holding between two types of real mean value integrals. We raise the question of whether the two types of mean value integrals can be shown to possess different growth rates for special choices of polynomial phases.
	
	\section{Mean values on sparse subsets of \texorpdfstring{$[0,1]^k$}{[0,1]k}}\label{sec:padic}

    In this section, we establish a transference principle between short mean value estimates over $\frak{p}$-adic fields and the reals, for phases given by a vector of homogeneous polynomials with integer coefficients. Because the domains of the integrals are smaller than the fundamental domain, we do not quite have access to a clean sum-product expansion. One approach would be to appeal to a more general sum-product quantity involving near-solutions to the associated Diophantine equations; we instead convert directly between the values of the integrands.
    
    Let $\frak{p}\in\N$ be an arbitrary prime. Write $\chi_\frak{p}$ for the standard $\frak{p}$-adic character, i.e.
    \begin{equation*}
        \chi_\frak{p}\left(\sum_{n=a}^\infty c_n\frak{p}^n\right)=e\left(\sum_{n=a}^{-1}c_n\frak{p}^n\right),
    \end{equation*}
    for each $a\in\Z$ and coefficients $c_n\in\{0,\ldots,\frak{p}-1\}$. Note in particular that $\chi_\frak{p}(q)=e(q)$ for any $q\in\Q$.
    
    Let $N=\frak{p}^K$ be a scale parameter coinciding with a power of $\frak{p}$. Let $\Omega\subseteq\Z^d$ be finite, usually depending on $N$. Let $\bb P:\R^d\to\R^k$ be a polynomial whose entries are homogeneous polynomials with rational coefficients. We write $\bb{P}=(\bb P_1,\ldots,\bb P_k)$ and $\bb P_j(\xi)=\sum_\ell c_j^\ell\xi^{e_j^\ell}$ for $\xi\in\R^d$ and $e_j^\ell=(e_{j1}^\ell,\ldots,e_{jd}^\ell)$ multiindices and $c_j^\ell\in\Z$. We write $|e_j^\ell|=\sum_i|e_{ji}^\ell|$; note that homogeneity implies that $|e_j^\ell|$ is independent of $\ell$, so we will frequently suppress that index as $|e_j^\ell|=|e_j|$.

    Finally, let $2\leq r<\infty$ and $\sigma=(\sigma_1,\ldots,\sigma_k)\in\R_{\geq 0}^k$ be a vector of nonnegative reals, serving as localization parameters. For simplicity, we assume that $\sigma_jK\in\Z$ for every $j$.
    \begin{definition}
        The $\frak{p}$-adic discrete restriction constant at exponent $r$, scale $N$, and localization $\sigma$ for the phase $\bb P$ over $\Omega$ is the optimal constant $D_r^\frak{p}(N,\sigma;\bb P,\Omega)$ validating the inequality
        \begin{equation*}
            N^{\sum_j\sigma_j}\int_{B_{\Q_p}(0,N^{-\sigma_1})\times\cdots\times B_{\Q_p}(0,N^{-\sigma_k})}\left|\sum_{\bf n\in\Omega}a_{\bf n}\chi_p\Big(\bf x\cdot\bb P\Big(\frac{\bf n}{N}\Big)\Big)\right|^rd\bf x\leq D_r^\frak p(N,\sigma;\bb P,\Omega)\sum_{\bf n\in\Omega}|a_{\bf n}|^r,
        \end{equation*}
    uniformly in $\{a_\bf n\}_{\bf n\in\Omega}$ choices of complex coefficients.
    \end{definition}

    Bounds on $D_r^\frak p(N,\sigma;\bb P,\Omega)$ imply bounds on mean values of exponential sums on special subdomains of $[0,1]^k$. We establish a shorthand for those subsets.
    \begin{definition}\label{def:subdoms}
        Given the datum $(\frak p,N,\bb P,\sigma)$, we define
        \begin{equation*}
            A_\frak p^{(N,\sigma;\bb P)}=\bigcup_{\substack{(\iota_1,\ldots,\iota_k)\\0\leq\iota_j<N^{|e_j|-\sigma_j}}}\Big(\iota_1N^{\sigma_1-|e_1|},\ldots,\iota_kN^{\sigma_k-|e_k|}\Big)+\prod_{j=1}^k\Big[-\frac{1}{2}N^{-|e_j|},\frac{1}{2}N^{-|e_j|}\Big]+\Z^k,
        \end{equation*}
        the \emph{sparse subdomain of $\R^k/\Z^k$ for $\bb P$ with $\frak p$-adic localization $(N,\sigma)$}.
    \end{definition}
    Here the sum of sets is the Minkowski sum. Note that $A_\frak p^{(N,\sigma;\bb P)}$ depends on $\bb P$ only through the homogeneity parameters $|e_1|,\ldots,|e_k|$. Note too that $|A_\frak p^{(N,\sigma;\bb P)}|=N^{-\sum_j\sigma_j}$. The sets are illustrated as the (b) domains in Figures \ref{fig:fig1}, \ref{fig:fig2}, and \ref{fig:fig3} below.
    
    \begin{definition}
        The real discrete restriction constant at exponent $r$, scale $N\in\frak p^\N$, and localization $\sigma$ for the phase $\bb P$ over $\Omega$ is the optimal constant $D_r^\infty(N,\sigma;\bb P,\Omega)$ validating the inequality
        \begin{equation*}
            N^{\sum_j\sigma_j}\int_{A_\frak p^{(N,\sigma;\bb P)}}\left|\sum_{\bf n\in\Omega}a_{\bf n}e\Big(\bf x\cdot\bb P(\bf n)\Big)\right|^rd\bf x\leq D_r^\infty(N,\sigma;\bb P,\Omega)\sum_{\bf n\in\Omega}|a_{\bf n}|^r,
        \end{equation*}
    uniformly in $\{a_\bf n\}_{\bf n\in\Omega}$ choices of complex coefficients.
    \end{definition}
    \begin{remark}
        It is most natural to take $\Omega=[0,N]^d$; this will in particular lead to an approximate matching between the two bounds \eqref{bd:abyn} and \eqref{bd:nbya} below. We leave the statement in its more general form, however.
    \end{remark}

    Our first result is the following. The proof is primarily a change-of-variables.
    \begin{theorem}[Transference of mean value estimates]\label{thm:transfer}
        Let $\frak p$, $\bb P$, $N$, $r$, and $\Omega$ be as above. Suppose $\sigma_j\leq|e_j|$ for all $j$. Then we have the bound
        \begin{equation}\label{bd:abyn}
            D_r^\infty(N,\sigma;\bb P,\Omega)\leq D_r^\frak p(N,\sigma;\bb P,\Omega).
        \end{equation}
        If $\epsilon_j=\max(1,\max_{\bf n\in\Omega}|\sum_\ell c_j^\ell(\bf n/N)^{e_j^\ell}|)^{-1}$, then
        \begin{equation}\label{bd:nbya}
            D_r^\frak p(N,\sigma;\bb P,\Omega)\leq \frac{2^{(r+1)k}}{\prod_{j=1}^{k}\epsilon_{j}}D_r^\infty(N,\sigma;\bb P,\Omega).
        \end{equation}
    \end{theorem}
    
    \begin{proof}
    We first establish the basic change-of-variable identity. Let $\{a_\bf n\}_{\bf n\in\Omega}$ be arbitrary. By local constancy considerations,
    \begin{equation*}
        \begin{split}
            &N^{\sum_j\sigma_j}\int_{B_{\Q_\frak p}(0,N^{-\sigma_1})\times\cdots\times B_{\Q_\frak p}(0,N^{-\sigma_k})}\left|\sum_{\bf n\in\Omega}a_{\bf n}\chi_\frak p\Big(\bf x\cdot\bb P\Big(\frac{\bf n}{N}\Big)\Big)\right|^rd\bf x\\
            &=N^{\sum_j(\sigma_j-|e_j|)}\sum_{\substack{(\iota_1,\ldots,\iota_k)\\0\leq\iota_j<N^{|e_j|-\sigma_j}}}\left|\sum_{\bf n\in\Omega}a_{\bf n}\chi_\frak p\Big(\sum_{j=1}^k\iota_j N^{\sigma_j-|e_j|}\sum_\ell c_j^\ell \bf n^{e_j^\ell}\Big)\right|^r.
        \end{split}
    \end{equation*}
    As the points of evaluation for the $\chi_\frak p$ are all rational, we may freely replace each $\chi_\frak p(\cdot)$ with $e(\cdot)$. Thus, we have the identity
    \begin{equation}\label{eq:cov}
        \begin{split}
            &N^{\sum_j\sigma_j}\int_{B_{\Q_\frak p}(0,N^{-\sigma_1})\times\cdots\times B_{\Q_\frak p}(0,N^{-\sigma_k})}\left|\sum_{\bf n\in\Omega}a_{\bf n}\chi_\frak p\Big(\bf x\cdot\bb P\Big(\frac{\bf n}{N}\Big)\Big)\right|^rd\bf x\\
            &=N^{\sum_j(\sigma_j-|e_j|)}\sum_{\substack{(\iota_1,\ldots,\iota_k)\\0\leq\iota_j<N^{|e_j|-\sigma_j}}}\left|\sum_{\bf n\in\Omega}a_{\bf n}e\Big(\sum_{j=1}^k\iota_j N^{\sigma_j-|e_j|}\sum_\ell c_j^\ell \bf n^{e_j^\ell}\Big)\right|^r.
        \end{split}
    \end{equation}
    Before proceeding to the two bounds \eqref{bd:abyn} and \eqref{bd:nbya}, we briefly explain the idea. The change-of-variable \eqref{eq:cov} equates the $\frak p$-adic short mean value integral to the sum of values of a real exponential sum, over the leftmost summand in the display of Definition \ref{def:subdoms}. Thus, we obtain a discrete version of the desired transference result. The two inequalities then follow by averaging and Fourier uncertainty.
    
    We now demonstrate \eqref{bd:abyn}. Write
    \begin{equation*}
        a_{\bf n}(v)=a_\bf ne\Big(\sum_{j=1}^kv_j\sum_\ell c_j^\ell\bf  n^{e_j^\ell}\Big),\quad v\in\R^k,
    \end{equation*}
    so that $a_{\bf n}(0)=a_{\bf n}$ and $|a_{\bf n}(v)|=|a_{\bf n}|$ for all $v$. Let $\cal R=\prod_{j=1}^k[-\frac{1}{2}N^{-|e_j|},\frac{1}{2}N^{-|e_j|}]$.
    Then we have
    \begin{equation*}
        \begin{split}
            &N^{\sum_j\sigma_j}\int_{A_\frak p^{(N,\sigma;\bb P)}}\left|\sum_{\bf n\in\Omega}a_{\bf n}e\big(\bf x\cdot\bb{P}(\bf n)\big)\right|^rd\bf x\\
            &=N^{\sum_j|e_j|}\int_{\cal R}N^{\sum_j(\sigma_j-|e_j|)}\sum_{\substack{(\iota_1,\ldots,\iota_k)\\0\leq\iota_j<N^{|e_j|-\sigma_j}}}\left|\sum_{\bf n\in\Omega}a_n(v)e\Big(\sum_{j=1}^k\iota_jN^{\sigma_j-|e_j|}\sum_\ell c_j^\ell\bf n^{e_j^\ell}\Big)\right|^rdv\\
            &=N^{\sum_j(\sigma_j+|e_j|)}\int_{\cal R}\int_{B_{\Q_\frak p}(0,N^{-\sigma_1})\times\cdots\times B_{\Q_\frak p}(0,N^{-\sigma_k})}\left|\sum_{\bf n\in\Omega}a_{\bf n}(v)\chi_\frak p\Big(\bf x\cdot\bb P\Big(\frac{\bf n}{N}\Big)\Big)\right|^rd\bf xdv.
        \end{split}
    \end{equation*}
    By the definition of $D_r^\frak p(N,\sigma;\bb P,\Omega)$ and the identity $|a_\bf n(v)|=|a_\bf n|$ for all $v\in\cal R$, we conclude
    \begin{equation*}
        N^{\sum_j\sigma_j}\int_{A_\frak p^{(N,\sigma;\bb P)}}\left|\sum_{\bf n\in\Omega}a_{\bf n}e\big(\bf x\cdot\bb P(\bf n)\big)\right|^rd\bf x\leq D_r^\frak p(N,\sigma;\bb P,\Omega)\sum_{\bf n\in\Omega}|a_{\bf n}|^r.
    \end{equation*}
    The inequality \eqref{bd:abyn} follows. We now consider \eqref{bd:nbya}. Picking up from the right-hand side of \eqref{eq:cov}, and adopting the abbreviation
    \begin{equation*}
        b_\bf n=a_\bf nN^{\sum_j|e_j|}\prod_j(\epsilon_j/2)^{-1}\prod_{j=1}^k\frac{\frac{\epsilon_j}{2}\pi\sum_{\ell}c_j^\ell(\bf n/N)^{e_j^\ell}}{\sin\Big(\frac{\epsilon_j}{2}\pi \sum_{\ell}c_j^\ell(\bf n/N)^{e_j^\ell}\Big)},\quad\bf n\in\Omega,
    \end{equation*}
    we evaluate
    \begin{equation*}
        \begin{split}
            &N^{\sum_j(\sigma_j-|e_j|)}\sum_{\substack{(\iota_1,\ldots,\iota_k)\\0\leq\iota_j<N^{|e_j|-\sigma_j}}}\left|\sum_{\bf n\in\Omega}a_\bf ne\Big(\sum_{j=1}^kN^{\sigma_j-|e_j|}\iota_j\sum_\ell c_j^\ell \bf n^{e_j^\ell}\Big)\right|^r\\
            &=N^{\sum_j(\sigma_j-|e_j|)}\sum_{\substack{(\iota_1,\ldots,\iota_k)\\0\leq\iota_j<N^{|e_j|-\sigma_j}}}\left|\sum_{\bf n\in\Omega}b_\bf ne\Big(\sum_{j=1}^k(\cdot)_j\sum_\ell c_j^\ell\bf n^{e_j^\ell}\Big)*1_{\frac{\epsilon}{2}\cal R}\Big(N^{\sigma_1-|e_1|}\iota_1,\ldots,N^{\sigma_k-|e_k|}\iota_k\Big)\right|^r.
        \end{split}
    \end{equation*}
    Here and on we write $\frac{\epsilon}{2}\cal R$ for the output of scaling the $j^{th}$ coordinate of $\cal R$ by $\frac{\epsilon_j}{2}$. By H\"older, the latter is bounded by
    \begin{equation*}
        (\prod_j\epsilon_j/2)^{r-1}N^{-(r-1)\sum_j|e_j|}\int_{\frac{\epsilon}{2}\cal R}N^{\sum_j(\sigma_j-|e_j|)}\sum_{\substack{(\iota_1,\ldots,\iota_k)\\0\leq\iota_j<N^{|e_j|-\sigma_j}}}\left|\sum_{\bf n\in\Omega}b_ne\Big(\sum_{j=1}^k\big(\iota_jN^{\sigma_j-|e_j|}-\bf x_j\big)\sum_\ell c_j^\ell \bf n^{e_j^\ell}\Big)\right|^rd\bf x.
    \end{equation*}
    We push the $r$ powers inside of the absolute values, and use $\epsilon_j/2<1$ to expand the domain; the last display is then bounded by
    \begin{equation*}
        (\prod_j\epsilon_j/2)^{-1}N^{\sum_j\sigma_j}\int_{A_\frak p^{(N,\sigma;\bb P)}}\left|\sum_{\bf n\in\Omega}(\prod_j\epsilon_j/2)N^{-\sum_j|e_j|}b_{\bf n}e\big(\bf x\cdot\bb{P}(\bf n)\big)\right|^rd\bf x.
    \end{equation*}
    Applying the definition of $D_r^\infty(N,\sigma;\bb P,\Omega)$, and using the smallness assumption on $\epsilon_j$ to conclude a lower bound on the cardinal sine, we conclude the upper bound
    \begin{equation*}
        \frac{2^{(r+1)k}}{\prod_{j=1}^{k}\epsilon_{j}}D_r^\infty(N,\sigma;\bb P,\Omega)\sum_{\bf n\in\Omega}|a_\bf n|^r.
    \end{equation*}
    Unrolling the equalities and bounds, we conclude \eqref{bd:nbya}.
    \end{proof}

    \begin{remark}
        In the special case of $r=2s$ for a positive integer $s$, the inequality \eqref{bd:abyn} can be obtained up to a constant loss by sum-product analysis and introducing technical weights; by operator interpolation, one concludes similar results for all $r\geq 2$. However, our method brings this constant down to $1$.
    \end{remark}

    We illustrate the domains with a few special cases.

    \begin{example}
        Let $\bb P:\R\to\R^2$ be $\bb P(x)=(x,x^2)$ and $\sigma=(0,\sigma_2)$. Then, for any choice of $\frak p$ and $N=\frak p^K$ for which $\sigma_2K\in\N$, the set $A_\frak p^{(N,\sigma;\bb P)}$ is a union of $N^{2-\sigma_2}$-many strips of shape $[0,1]\times[0,N^{-2}]$, equally spaced in the vertical interval $[0,1]$. See Figure \ref{fig:fig1}.

        \begin{figure}
            \centering
            \begin{subfigure}{.5\textwidth}
              \centering
              \includegraphics[width=.6\linewidth]{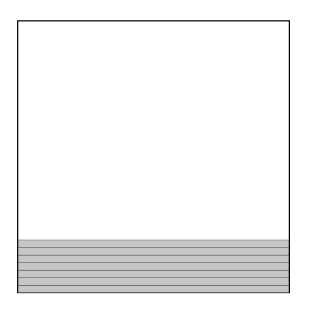}
              \caption{Short subdomain of $[0,1]^2$}
              \label{fig:fig1sub1}
            \end{subfigure}%
            \begin{subfigure}{.5\textwidth}
              \centering
              \includegraphics[width=.6\linewidth]{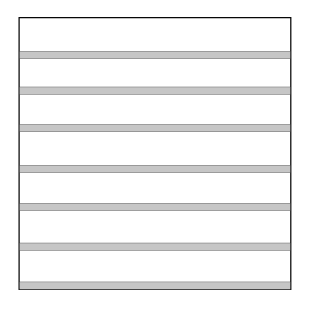}
              \caption{Sparse subdomain of $[0,1]^2$}
              \label{fig:fig1sub2}
            \end{subfigure}
            \caption{Domains for discrete restriction for the parabola in $\R^2$.}
            \label{fig:fig1}
        \end{figure}
        
    \end{example}

    \begin{example}
        Let $\bb P:\R\to\R^3$ be $\bb P(x)=(x,x^2,x^3)$ and $\sigma=(0,0,\sigma_3)$. Then, for any choice of $\frak p$ and $N=\frak p^K$ for which $\sigma_3K\in\N$, the set $A_\frak p^{(N,\sigma;\bb P)}$ is a union of $N^{3-\sigma_3}$-many plates of shape $[0,1]^2\times[0,N^{-3}]$, equally spaced in the vertical interval $[0,1]$. See Figure \ref{fig:fig2}.

        \begin{figure}
            \centering
            \begin{subfigure}{.5\textwidth}
              \centering
              \includegraphics[width=.6\linewidth]{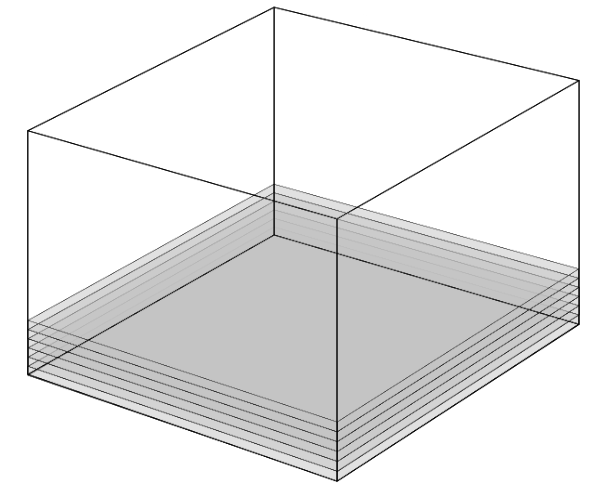}
              \caption{Short subdomain of $[0,1]^3$}
              \label{fig:fig2sub1}
            \end{subfigure}%
            \begin{subfigure}{.5\textwidth}
              \centering
              \includegraphics[width=.6\linewidth]{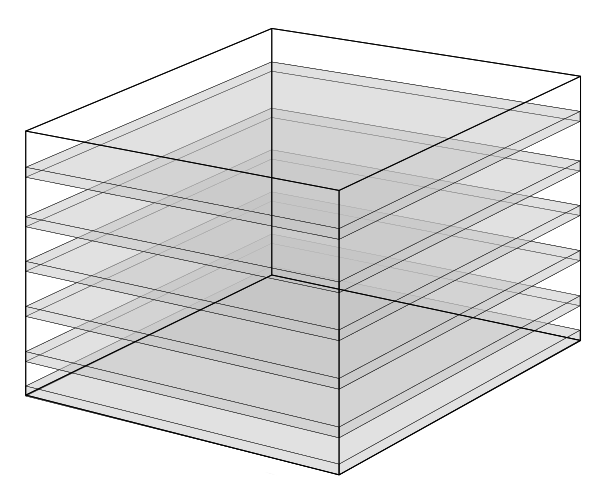}
              \caption{Sparse subdomain of $[0,1]^3$}
              \label{fig:fig2sub2}
            \end{subfigure}
            \caption{Domains for discrete restriction for the moment curve in $\R^3$.}
            \label{fig:fig2}
        \end{figure}

    \end{example}

    \begin{example}

        Let $\bb P:\R^2\to\R^3$ be $\bb P(x_1,x_2)=(x_1,x_2,x_1^2+x_2^2)$ and $\sigma=(0,\sigma_2,\sigma_3)$. Then, for any choice of $\frak p$ and $N=\frak p^K$ for which $\sigma_2K,\sigma_3K\in\N$, the set $A_\frak p^{(N,\sigma;\bb P)}$ is a union of planks of shape $[0,1]\times[0,N^{-1}]\times[0,N^{-2}]$, arranged in a grid whose $y$-displacement is $N^{\sigma_2-1}$ and whose $z$-displacement is $N^{\sigma_3-2}$, for a total of $N^{3-\sigma_2-\sigma_3}$ planks. See Figure \ref{fig:fig3}.

        \begin{figure}
            \centering
            \begin{subfigure}{.5\textwidth}
              \centering
              \includegraphics[width=.6\linewidth]{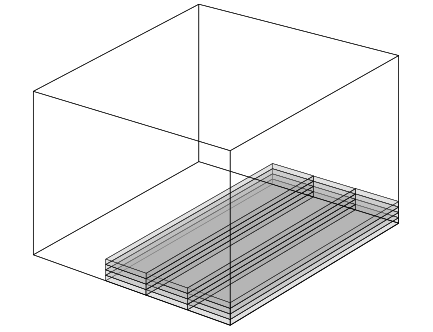}
              \caption{Short subdomain of $[0,1]^3$}
              \label{fig:fig3sub1}
            \end{subfigure}%
            \begin{subfigure}{.5\textwidth}
              \centering
              \includegraphics[width=.6\linewidth]{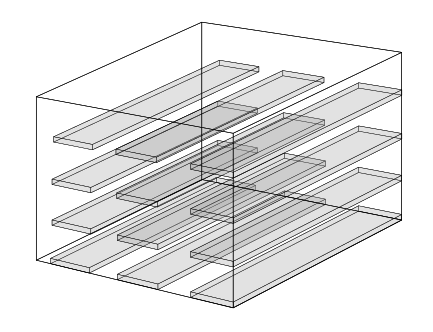}
              \caption{Sparse subdomain of $[0,1]^3$}
              \label{fig:fig3sub2}
            \end{subfigure}
            \caption{Domains for discrete restriction for the paraboloid in $\R^3$.}
            \label{fig:fig3}
        \end{figure}
        
    \end{example}

    In one of these examples, suitable small cap decoupling theorems are available to imply the mean value estimates in question.
    \begin{corollary}
        Let $\frak p$ be a prime and $N=\frak p^K$ for some $K$. Let $0\leq\sigma\leq 1$. Then, for any choice of $\{a_k\}_{0\leq k<N}$ complex and any exponent $r\geq 2$, we have
        \begin{equation*}
            N^{\sigma}\int_{A_\frak p^{(N,(0,\sigma);(x,x^2))}}\left|\sum_{n=0}^{N-1}a_ne(sn+tn^2)\right|^rdsdt\lesssim\bf p^{12}(\log N)^{22}(N^{\frac{r}{2}}+N^{r-4+\sigma})\sum_{n=0}^{N-1}|a_n|^r
        \end{equation*}
    \end{corollary}
    \begin{proof}
        Combine Theorem 1 of \cite{JhighlowNA} and Theorem \ref{thm:transfer}, and an elementary argument.
    \end{proof}

    \begin{remark}\label{rmk:transfer}

    The general theme we observe is the following. Existing (real-number) small cap decoupling theorems imply good estimates for mean value estimates of exponential sums over the domains (a) in each of the above examples/figures. Decoupling proofs are often identical between real and $\frak p$-adic fields. Thus, one expects matching mean value estimates for the versions of the (a) domains in the $\frak p$-adic fields. By Theorem \ref{thm:transfer}, one concludes matching estimates for the ``sparse'' (b) domains. Thus, mean value estimates over the (a) domains \emph{morally} imply matching mean value estimates over the (b) domains.

    An important counterexample to the above theme is recorded below in Section \ref{section:obstruction}.
        
    \end{remark}

    We also point out that the transference principle also applies to the canonical-scale integrals, where the Archimedean and non-Archimedean shapes coincide; this amounts to choosing $\sigma=0$ in Theorem \ref{thm:transfer}. Thus:
    \begin{corollary}\label{cor:moment}
        Let $n\geq 2$ and $\frak p$ prime. Then, for any $N=\frak p^K$ and any choice of $\{a_k\}_{0\leq k<N}$ complex, and any exponent $r\geq 2$, we have
        \begin{equation*}
            \int_{[0,1]^n}\left|\sum_{k=0}^{N-1}a_ke(\bf x\cdot(t,t^2,\ldots,t^n))\right|^rd\bf x\leq\exp\Big(rC_n(\log\frak p)(\log N)^{1-\frac{1}{4n\log n+1}}\Big)\big(1+N^{\frac{r}{2}-n(n+1)}\big)\left(\sum_{k=0}^{N-1}|a_k|^2\right)^{r/2},
        \end{equation*}
        for some constant $C_n$ depending only on $n$.
    \end{corollary}
    \begin{proof}
        Combine Theorem \ref{thm:transfer} with \cite{johnsrude2024restricted}, Theorem 6.1.
    \end{proof}
    \begin{remark}
        This result, with the exponential at the front replaced with $C_\eps N^\eps$, has been known since \cite{bourgain2016proof}. A much stronger result in the dimension $n=2$ was demonstrated by $\frak p$-adic methods in \cite{GLY}; note that the argument there relies on using an even integer exponent and a sum-product expansion. Here, our argument is appropriate for real $r\geq 2$.
    \end{remark}

    \section{Integrals arising from decoupling over completions of algebraic fields}\label{sec:numfield}

    We now consider consequences of the decoupling theorem for the moment curve over non-Archimedean completions of algebraic fields. Let $P(x)=x^d+c_{d-1}x^{d-1}+\ldots+c_0$ be a monic irreducible polynomial with rational coefficients, and let $\alpha\in\C$ be some root of $P$. Let $L=\Q(\alpha)$ and write $\L$ be some non-Archimedean completion of $L$. Then the decoupling theorem for the moment curve in $\L^d$ \cite{johnsrudethesis} implies that, for arbitrary complex coefficients $a_\bf n$ and $r\leq k(k+1)$,
    \begin{equation*}
        \begin{split}
            \int_{\O_\L^k}&\left|\sum_{\substack{\bf n=(n_0,\ldots,n_{d-1})\\0\leq n_\iota<N\,\forall\iota}} a_\bf ne\left(\mathrm{Tr}_{\L/\Q}\Big[\bf x\cdot\Big\{\big((n_0/N)+(n_1/N)\alpha+\ldots+(n_{d-1}/N)\alpha^{d-1}\big)^j\Big\}_{j=1}^k\Big]\right)\right|^rd\bf x\\
            &\leq \exp\left(rC_{k,\alpha}(\log N)^{1-c_k}\right)\left(\sum_{\substack{\bf n=(n_0,\ldots,n_{d-1})\\0\leq n_\iota<N\,\forall\iota}} |a_\bf n|^2\right)^{r/2}.
        \end{split}
    \end{equation*}
    Note carefully that $e\circ\mathrm{Tr}_{\L/\Q}$ is a fundamental character for the field $\L$; note too that $\mathrm{Tr}_{\L/\Q}$ is $\Q$-linear. The left-hand side may be written with multinomial coefficients as
    \begin{equation*}
        N^{-\frac{dk(k+1)}{2}}\sum_{\substack{\bf w=(w_{j\ell})\\0\leq w_{j\ell}<N^j}}\left|\sum_{\substack{\bf n=(n_0,\ldots,n_{d-1})\\0\leq n_\iota<N\,\forall\iota}} a_\bf ne\left(\sum_{j=1}^k\sum_{\ell=0}^{d-1}\frac{w_{j\ell}}{N^j}\sum_{\substack{e_0+\ldots+e_{d-1}=j\\e_\iota\geq 0\,\forall\iota}}\binom{j}{e_0,\ldots,e_{d-1}}n_0^{e_0}\cdots n_{d-1}^{e_{d-1}}\mathrm{Tr}_{L/\Q}\big[\alpha^{\ell+\sum_{\iota=1}^{d-1}\iota e_\iota}\big]\right)\right|^r.
    \end{equation*}
    By following the same arguments in Theorem \ref{thm:transfer}, we obtain:
    \begin{theorem}\label{thm:mv_nf}
        Let $\alpha\in\overline{\Q}^{\mathrm{alg}}$ be algebraic of degree $d$, $L=\Q(\alpha)$, $k\in\N$, and $2\leq r\leq k(k+1)$. Then we have the mean value estimate
        \begin{equation}\label{ineq:mv_nf}
        \begin{split}
            \int_{[0,1]^{dk}}&\left|\sum_{\substack{\bf n=(n_0,\ldots,n_{d-1})\\0\leq n_\iota<N\,\forall\iota}} a_\bf ne\left(\sum_{j=1}^k\sum_{\ell=0}^{d-1}\bf x_{j\ell}\sum_{\substack{e_0+\ldots+e_{d-1}=j\\e_\iota\geq 0\,\forall\iota}}\mathrm{Tr}_{L/\Q}\big[\alpha^{\ell+\sum_{\iota=1}^{d-1}\iota e_\iota}\big]\binom{j}{e_0,\ldots,e_{d-1}}n_0^{e_0}\cdots n_{d-1}^{e_{d-1}}\right)\right|^rd\bf x\\
            &\leq \exp\left(rC_k(d\log N)^{1-c_k}\right)\left(\sum_{\substack{\bf n=(n_0,\ldots,n_{d-1})\\0\leq n_\iota<N\,\forall\iota}} |a_\bf n|^2\right)^{r/2}.
        \end{split}
    \end{equation}
    \end{theorem}
    
    We supply several special examples.
    \begin{example}
        Let $P(x)=x^2+1$, $k\geq 1$, and $r=k(k+1)$. Then
        \begin{equation*}
        \begin{split}
            \int_{[0,1]^{2k}}&\left|\sum_{0\leq n_0,n_1<N} a_\bf ne\left(\sum_{j=1}^k\sum_{\ell=0}^{1}\bf x_{j\ell}\sum_{e_1=0}^j\epsilon_{\ell,e_1}\binom{j}{e_1}n_0^{j-e_1}n_{1}^{e_1}\right)\right|^{k(k+1)}d\bf x\\
            &\leq \exp\left(C_k'(\log N)^{1-c_k}\right)\left(\sum_{0\leq n_0,n_1<N} |a_\bf n|^2\right)^{k(k+1)/2},
        \end{split}
    \end{equation*}
    where
    \begin{equation*}
        \epsilon_{\ell,e_1}=\begin{cases}
            0 & \ell+e_1\,\,\,\text{odd},\\
            -1 & \ell+e_1\equiv 2\,\,\text{mod 4},\\
            1 & \ell+e_1\equiv 0\,\,\text{mod 4}.
        \end{cases}
    \end{equation*}
    In the case $k=3$, this takes the form
    \begin{equation*}
        \begin{split}
            \int_{[0,1]^{6}}&\left|\sum_{0\leq n_0,n_1<N} a_\bf ne\left(n_0\bf x_{10}-n_1\bf x_{11}+(n_0^2-n_1^2)\bf x_{20}-2n_0n_1\bf x_{21}+(n_0^3-n_0n_1^2)\bf x_{30}-(n_0^2n_1-n_1^3)\bf x_{31}\right)\right|^{12}d\bf x\\
            &\leq \exp\left(C(\log N)^{1-c}\right)\left(\sum_{0\leq n_0,n_1<N} |a_\bf n|^2\right)^{6},
        \end{split}
    \end{equation*}
    
    Returning to the general $k$ case, and specializing to the case that $a_{\bf n}=0$ unless $n_1=0$, we recover exactly the Bourgain--Demeter--Guth result on the main conjecture of the Vinogradov mean value theorem at the critical exponent.
    \end{example}
    \begin{example}
        Let $P(x)=x^3-2$ and $k=3$. Then
        \begin{equation*}
            \begin{split}
                \int_{[0,1]^{9}}&\left|\sum_{0\leq n_0,n_1,n_2<N} a_\bf ne\left(\bf x\cdot\Psi(\bf n)\right)\right|^{12}d\bf x\leq \exp\left(C(\log N)^{1-c}\right)\left(\sum_{0\leq n_0,n_1,n_2<N} |a_\bf n|^2\right)^{6},
            \end{split}
        \end{equation*}
        where the phase function $\Psi$ is given by
        \begin{equation*}
            \begin{split}\Psi(n_0,n_1,n_2)=\Big(&n_0,\,\,n_1,\,\,n_2,\\
            &n_0^2+4n_1n_2,\,\,n_1^2+2n_0n_2,\,\,n_2^2+n_0n_1,\\
            &4n_2^3+2n_1^3+12n_0n_1n_2+n_0^3,\,\,
            2n_1n_2^2+n_0n_1^2+n_0^2n_2,\,\,2n_1^2n_2+2n_0n_2^2+n_0^2n_1\Big).
            \end{split}
        \end{equation*}
    \end{example}
    \begin{remark}
        In the quadratic examples of Theorem \ref{thm:mv_nf}, i.e. where $k=2$, we are essentially identifying a family of graphs of tuples of quadratic forms which exhibit $\delta^{-\eps}$-fine decoupling up to $\ell^2L^6$; one may compare with the systematic analysis of decoupling over quadratic forms in \cite{GZZ}.
    \end{remark}
    
    We conclude by demonstrating a simple consequence for Vinogradov systems of algebraic numbers.
    \begin{corollary}\label{cor:alg_sol}
        Let $\alpha\in\overline{\Q}^{\mathrm{alg}}$ be algebraic of degree $d\geq 1$. Let $J_{s,k,d}(N;\alpha)$ be the number of solutions to the simultaneous system of $k$ equations
        \begin{equation*}
            \begin{split}
            \sum_{j=1}^s\sum_{\ell=0}^{d-1} n_{j\ell}\alpha^\ell&=\sum_{j=1}^s\sum_{\ell=0}^{d-1} m_{j\ell}\alpha^\ell\\
            \sum_{j=1}^s\Big(\sum_{\ell=0}^{d-1} n_{j\ell}\alpha^\ell\Big)^2&=\sum_{j=1}^s\Big(\sum_{\ell=0}^{d-1} m_{j\ell}\alpha^\ell\Big)^2\\
            &\,\,\,\vdots\\
            \sum_{j=1}^s\Big(\sum_{\ell=0}^{d-1} n_{j\ell}\alpha^\ell\Big)^k&=\sum_{j=1}^s\Big(\sum_{\ell=0}^{d-1} m_{j\ell}\alpha^\ell\Big)^k,
            \end{split}
        \end{equation*}
        among choices of $n_{j\ell},m_{j\ell}\in[0,N)$, with $1\leq j\leq s$ and $0\leq\ell\leq d-1$. Then we have the bound
        \begin{equation*}
            J_{s,k,d}(N;\alpha)\lesssim_{\eps,\alpha} N^\eps(N^{ds}+N^{2ds-dk(k+1)/2}).
        \end{equation*}
    \end{corollary}
    \begin{proof}
        Applying Theorem \ref{thm:mv_nf} with $r=2s$ and $a_\bf n\equiv 1$, we obtain the same upper bound to the system with the function $\operatorname{Tr}_{\Q(\alpha)/\Q}$ applied to both sides. Passing to the original system above reduces the number of relations, so we conclude the desired upper bound.
    \end{proof}

    \begin{remark}
        The case $d=1$ is just the Bourgain--Demeter--Guth result. For $d\geq 1$ arbitrary and $\alpha$ transcendental, the same upper bound holds for $J_{s,k,d}(N;\alpha)$ subject to the natural interpretation of the latter quantity; indeed, passage from $\alpha$ to some $\alpha'\in\overline{\Q}^{\mathrm{alg}}$ of degree $d$ introduces relations, hence more solutions.
    \end{remark}

    \section{An obstruction for transference arguments}\label{section:obstruction}

    In this section, we record an example which runs contrary to the intuition described in Remark \ref{rmk:transfer}. As small cap decoupling theorems imply short mean value estimates, and as Theorem \ref{thm:transfer} implies that $\frak p$-adic short mean value estimates transfer to mean value estimates over sparse real domains, any such example would either need to (a) exploit the gap between the two estimates \eqref{bd:abyn} and \eqref{bd:nbya} in Theorem \ref{thm:transfer}, or (b) demonstrate a quantitative gap between real and $\frak p$-adic decoupling theorems. We take the latter approach.
    
    Let $\frak p=4\ell+1$ be prime. It is well-known that $\Q_\frak p$ has a square root of $-1$, i.e. there exists a sequence $b_0,b_1,\ldots\in\{0,\ldots,\frak p-1\}$ such that $\xi=\sum_{n=0}^\infty b_n\frak p^n$ satisfies $\xi^2=-1$. It follows that the decoupling theory of the $(2+1)$ paraboloid over the $\frak p$-adics is significantly different from that over the reals.

    To this end, write $\bb P(a,b)=(a,b,a^2+b^2)$. Let $k\geq 1$ and write $N=\frak p^k$. For $0\leq n<N$, we write
    \begin{equation*}
        f_n(\bf x)=1_{\frak p^{-2k}\Z_p^3}(\bf x)\chi_\frak p(\bf x\cdot (n\xi,n,0)).
    \end{equation*}
    Then $\hat{f}_n$ is supported in the $N^{-2}$ ball centered at $(n\xi,n,0)$ on the $(2+1)$ paraboloid over $\Z_\frak p^2$. For $r\geq 2$, we have
    \begin{equation*}
        \|f_n\|_{r}=N^{\frac{6}{r}},\quad\left\|\sum_{n=0}^{N-1}f_n\right\|_{r}\gtrsim N^{1+\frac{5}{r}},
    \end{equation*}
    so that
    \begin{equation*}
        \frac{\left\|\sum_{n=0}^{N-1}f_n\right\|_{r}}{\left(\sum_{n=0}^{N-1}\|f_n\|_r^2\right)^{1/2}}\gtrsim N^{\frac{1}{2}-\frac{1}{r}}.
    \end{equation*}
    Thus, if $\mathrm{Dec}_{\ell^2L^r}^{\Q_{\frak p}}(\delta,\bb P)$ is the decoupling constant associated to the standard partition of the $\delta^2$ neighborhood of the truncated paraboloid $\bb P$ over $\Q_\frak p$, then we have
    \begin{equation*}
        \mathrm{Dec}_{\ell^2L^r}^{\Q_\frak p}(\delta,\bb P)\gtrsim\delta^{-\frac{1}{2}+\frac{1}{r}}.
    \end{equation*}
    In contrast, the Bourgain--Demeter theorem for the paraboloid supplies the bound for the \emph{real} decoupling constant
    \begin{equation*}
        \mathrm{Dec}_{\ell^2L^r}^\R(\delta,\bb P)\lesssim_\eps\delta^{\eps},\quad\forall\eps>0,
    \end{equation*}
    in the range $2\leq r\leq 4$.

    We are not aware of a corresponding example for discrete restriction constants. Note that the canonical-scale discrete restriction constants always coincide between the reals and $\frak p$-adics, so such an example would have to come from short/sparse integrals. As these are proved by small cap decoupling theorems, one would need to leverage some gap between small cap decoupling theorems of the two settings. On the other hand, most small cap decoupling theorems are proven in the $\ell^pL^p$ context, where we do not anticipate any difference in the decoupling theory. The only examples the author is aware of in $\ell^qL^p$ small cap decoupling with $q<p$ are for special cases of the moment curve, where no major differences between the reals and the $\frak p$-adics should appear; see \cite{FGM,johnsrude2023small,JhighlowNA}.

    \begin{problem}
        Find a choice of $\frak p$, $\bb P$, $q$, $r$, and $\sigma$ such that the $\ell^q\to L^r$ discrete restriction constant for $\Omega=\{0,\ldots,N\}^d$ over the domain $[0,N^{-\sigma_1}]\times\cdots\times[0,N^{-\sigma_k}]$ obeys a different polynomial dependence on $N$ than the corresponding constant for the domain $A_{\frak p}^{(N,\sigma;\bb P)}$.
    \end{problem}

    \begin{remark}
        The isotropy of the dot product suggests that $\frak p$-adic paraboloids should be treated as hyperbolic paraboloids. Moreover, in dimensions greater than 3, many more $\frak p$-adic fields have isotropy. It is interesting to note, in view of this phenomenon, that the moment curve exhibits no such breakdown in decoupling theory; see \cite{johnsrude2024restricted}, Theorem 6.1. The critical demonstration of this fact is the latter's Lemma 6.11, which demonstrates that projections of the moment curve still have curvature. See also that paper's Proposition 6.13.
    \end{remark}

	\printbibliography
	
\end{document}